\documentclass[11pt,reqno]{amsart}

\usepackage{amsmath,amsthm}

\usepackage{amsfonts,amssymb,amscd,enumerate,verbatim}

\usepackage[latin1]{inputenc}
\usepackage{latexsym}
\usepackage{hyperref}

\usepackage[shortlabels, inline]{enumitem}

\usepackage{graphicx}
\usepackage{xcolor}

\usepackage{mathptmx}
\def\NZQ{\mathbb}               %

\def\ZZ{{\NZQ Z}}
\def\RR{{\NZQ R}}

\def\frk{\mathfrak}               %

\def\Phi{{\frk N}}

\def\lamb{{\mathbf \lambda}}

\def\opn#1#2{\def#1{\operatorname{#2}}} %

\def\Cc{{\mathcal C}}

\def\Oc{{\mathcal O}}

\def\Cc{{\mathcal C}}

\newtheorem{thm}{Theorem}[section]
\newtheorem{lem}[thm]{Lemma}
\newtheorem{cor}[thm]{Corollary}
\newtheorem{prop}[thm]{Proposition}
\newtheorem{quest}[thm]{Question}

\theoremstyle{definition}

\let\epsilon\varepsilon
\let\phi=\varphi
\let\kappa=\varkappa
\opn\conv{conv}

\newcommand{\TP}[1][P]{\mathbb{DP}(#1)}%

\newcommand{\ol}[1]{\overline{#1}}%

\newcommand\Def[1]{\textbf{#1}}%

\author{Akihiro Higashitani}
\address[A.~Higashitani]{
Department of Pure and Applied Mathematics,
Graduate School of Information Science and Technology,
Osaka University,
Suita, Osaka 565-0871, Japan}
\email{higashitani@ist.osaka-u.ac.jp}

\author{Arnau Padrol}
\address[A.~Padrol]{Departament de Matem\`atiques i Inform\`atica, Universitat de Barcelona, Barcelona, Spain and Centre de Recerca Matem\`atica, Bellaterra, Spain.}
\email{arnau.padrol@ub.edu}

\author{Raman Sanyal}
\address[R.~Sanyal]{Institut f\"ur Mathematik,
Goethe-Universit\"at Frankfurt, Frankfurt am Main, Germany}
\email{sanyal@math.uni-frankfurt.de}

\thanks{
The research of A.H. is supported by JSPS KAKENHI Grant Numbers JP24K00521 and 21KK0043.
The collaboration of A.P. and R.S.~is supported by the
Spanish-German project COMPOTE (AEI PCI2024-155081-2 \& DFG 541393733).
The research of A.P. is also supported by projects PID2022-137283NB-C21 of MICIU/AEI/10.13039/501100011033 / FEDER, UE and PAGCAP
ANR-21-CE48-0020 \& FWF I 5788, as well as the Mar\'i­a de Maeztu program CEX2020-001084-M. Research of R.S.~is also supported by the DFG priority program \emph{Combinatorial Synergies} (SPP 2458) --
539866293, 539866395.}

\begin{document}

\title{Indecomposability of $0/1$-polytopes}

\subjclass[2020]{%
    Primary %
    52B12; %
    Secondary %
    52B20, %
    90C57, %
    05B35, %
    06A07} %
\keywords{0/1-polytope, Minkowski decomposition, indecomposable polytope, deformation polytope, cone, lattice polytope, stable set polytope, matroid polytope, order polytope, flow polytope}

\begin{abstract}
    We prove that every $0/1$-polytope has a unique Minkowski decomposition into
    indecomposable polytopes, up to translation of summands. The summands lie in
    pairwise orthogonal subspaces. Thus, every $0/1$-polytope is the Cartesian
    product of indecomposable $0/1$-polytopes.

    As applications, we obtain uniform combinatorial indecomposability criteria for order and chain
    polytopes, matroid polytopes, stable set and clique polytopes, edge polytopes, flow polytopes,
    and $2$-level/compressed polytopes. We also show that every nontrivial factorization of a
    multi-affine polynomial is a product of multi-affine polynomials in disjoint sets of variables.
\end{abstract}
\maketitle

\section{Introduction}
\noindent
The \Def{Minkowski sum} of two convex polytopes $Q,R \subset \RR^n$ is the polytope
\[
    Q + R \ = \ \{ q + r : q \in Q, r \in R\} \, .
\]
A polytope $P$ is \Def{(Minkowski) decomposable} if there are polytopes $Q,R$ such that $P = Q+R$
  and neither $Q$ nor $R$ are homothetic\footnote{By homothetic we mean related by translation and (positive) dilation.} to $P$. Minkowski decompositions are important in various
  areas. In this case $Q$ and $R$ are called proper (Minkowski) summands of $P$. Research has
  mainly focussed on the space of (homothety-classes) of summands, which are called \emph{type
  cones} in convex geometry~\cite{McMullen-reps,McMullen-simple}, \emph{nef cones} in algebraic geometry~\cite{CoxLittleSchenckToric,Altmann}, and \emph{deformation cones} in geometric combinatorics~\cite{Postnikov,PostnikovReinerWilliams,Ardila, DefConeCubes}.
These cones are notoriously difficult to understand, and special emphasis has been given to
identifying the rays (e.g. in \cite{Gale1954,Shephard,Kallay1982,McMullen1987,PrzeslawskiYost2008,PrzeslawskiYost2016,PP}).
As faces of deformation cones are deformation cones, the rays of deformation cones correspond to
(homothety-classes) of \Def{(Minkowski) indecomposable} polytopes. Indecomposability has been
established for various disparate classes of polytopes of interest in combinatorial optimization,
geometric/algebraic combinatorics, and toric geometry with ingenious arguments within the
respective context \cite{RosenmullerWeidner1973-ExtremeConvexSetFunctions,Ngu86,Ardila,Fourier2016,PPR,LPP}. In this paper we prove indecomposability for a large class of polytopes that
contains remarkably many families of polytopes of combinatorial interest. 

A polytope $P \subset \RR^n$ with vertices $V(P)$ is a \Def{0/1-polytope} if $V(P) \subseteq
\{0,1\}^n$. We call a polytope \Def{proper} if it is not a point.

\begin{thm}\label{thm:main}
    Every $0/1$-polytope $P$ decomposes uniquely as a Cartesian product of
    proper indecomposable $0/1$-polytopes. In particular, $P$ is either indecomposable or a product of proper $0/1$-polytopes.
\end{thm}

There are many invariants that behave well with respect to taking products and which directly give
sufficient conditions on decomposability. We refer to Section~\ref{sec:01-TP} for definitions or
references to undefined notions in the following list. Observe that by using coordinate projections
we can always assume that $P$ is full-dimensional.

\newcommand\Ehr{\mathrm{Ehr}}%
\begin{cor}
    Let $P \subset \RR^n$ be a full-dimensional $0/1$-polytope. Then $P$ is indecomposable if one of the following conditions holds:
    \begin{enumerate}[\rm(i)]
        \item The $f$-polynomial $f_P(t)$ is irreducible.
        \item The Ehrhart polynomial $\Ehr_P(n)$ is irreducible.
        \item The graph $G_P = ([n],E)$ with edges for every pair of coordinates $ij$ for which
            there is a facet normal $a$ with $a_ia_j \neq 0$ is connected.
    \end{enumerate}
\end{cor}

The proof of the main result as well as the necessary conditions are given in
Section~\ref{sec:01-TP}. We list the most prominent applications of Theorem~\ref{thm:main} and
defer definitions, references, and proofs to Section~\ref{sec:applications}.

\newcommand\Stab{\mathrm{Stab}}%
\newcommand\Cli{\mathrm{Cli}}%
\newcommand\Mat{\mathrm{Mat}}%
\newcommand{\EP}{\mathcal{P}}%
\newcommand{\poset}{\Pi}%
\newcommand{\Flow}{\mathrm{Flow}}%
\begin{cor}\label{cor:consequences}
    The following give combinatorial characterizations of indecomposability:
    \begin{enumerate}[\rm(1)]
        \item \label{item:Oc}%
            Let $(\poset,\preceq)$ be a finite poset. Then the order polytope $\Oc(\poset)$ is indecomposable
            if and only if $\poset$ is a connected poset.
        \item \label{item:Cc}%
            Let $(\poset,\preceq)$ be a finite poset. Then the chain polytope $\Cc(\poset)$ is
            indecomposable if and only if $\poset$ is a connected poset.
        \item \label{item:Stab}%
            Let $G = (V,E)$ be an undirected simple graph. Then the stable set polytope
            $\Stab(G)$ is indecomposable if and only if $G$ is connected.
        \item Let $G = (V,E)$ be an undirected simple graph. Then the clique polytope $\Cli(G)$ is
            indecomposable if and only if the complement graph $\overline{G}$ is connected.
        \item Let $G = (V,E)$ be a connected simple graph. Then the edge polytope $\EP_G$ is
            indecomposable.
        \item \label{item:matroid} Let $M$ be a matroid. The matroid base polytope $B_M$ is indecomposable if and only
            if the independent set polytope $P_M$ is indecomposable if and only if $M$ is
            connected. \item\label{item:perm} Let $H$ be a finite group. The permutation polytope
            $P(H)$ is indecomposable if and only if $H$ is not the product of two proper subgroups.
            In particular the Birkhoff polytope is indecomposable.
        \item Let $D$ be a directed acyclic graph with unique source $s$ and sink $t$. The flow
            polytope $\Flow_D(-1,0,\dots,0,1)$ is indecomposable if and only if $D$ does not have
            an $s-t$ separator of size $1$.

    \end{enumerate}
\end{cor}

Indecomposability of matroid base and independent set polytopes in \ref{item:matroid} was already known \cite{Ngu86,PP,LPP}. To the best of our knowledge, the other statements in this corollary are new.
Note that \ref{item:Cc} is a special case of \ref{item:Stab}. In (2)--(4) the polytopes in question are \emph{antiblocking} polytopes. 
For $0/1$-antiblocking polytopes as well as for order polytopes we give additional proofs. Whereas the proof of Theorem~\ref{thm:main} heavily relies on geometry, the additional proofs only rely on the underlying combinatorial structures.

\newcommand{\K}{\mathbb{K}}%
\newcommand{\Newt}{\mathrm{Newt}}%
We close this section with a simple but fundamental algebraic implication. Let $\K$ be a field and let $f \in \K[x_1,\dots,x_n]$ be a multivariate polynomial over $\K$. We can write $f$ as
\[
    f \ = \ \sum_{\alpha \in A} c_\alpha x_1^{\alpha_1} x_2^{\alpha_2} \cdots x_n^{\alpha_n}
\]
for some finite subset $A \subseteq \ZZ^n_{\ge0}$ and $(c_\alpha)_{\alpha \in A} \in \K$. The
  \Def{Newton polytope} of $f$ is the lattice polytope $\Newt(f) := \conv\{ \alpha \in A : c_\alpha
  \neq 0 \} \subset \RR^n$. If $f = gh$ for non-zero polynomials $g,h \in \K[x_1,\dots,x_n]$, then
  $\Newt(f) = \Newt(g) + \Newt(h)$. Hence if $f$ is reducible, then $\Newt(f)$ is decomposable.

A polynomial $f$ is \Def{multi-affine} if it is an affine-linear function in each variable.
Equivalently, $f$ is multi-affine if $\Newt(f)$ is a $0/1$-polytope. For $J \subseteq [n]$, let
$f_J$ be the restriction of $f$ to the terms involving only variables in $\{x_i : i \in J\}$. We
set $\ol{J} := [n] \setminus J$.

\begin{cor}
    Let $f \in \K[x_1,\dots,x_n]$ be a multi-affine polynomial. If $f$ is reducible, then there is $J \subset [n]$ such that $f = f_J f_{\ol{J}}$.
\end{cor}

\noindent
\textbf{Acknowledgments.}
AP wants to thank Germain Poullot for many interesting conversations on the topic. RS would like to
thank Katharina Jochemko and Benjamin Schr\"oter for fruitful discussions during the
MFO-Mini-Workshop \emph{Alcoved Polytopes in Physics and Optimization} (2025).

\section{Decomposability and deformation polytopes}

We keep the preliminaries brief and refer to \cite{Gruenbaum,ziegler} for background on Minkowski
sums and \cite{PP} for a recent and more comprehensive account on decomposability.

\newcommand{\inner}[1]{\langle {#1} \rangle}%
Let $P \subset \RR^d$ be a polytope with vertex set $V(P)$ and edge set $E
\subseteq \binom{V(P)}{2}$. For $c \in \RR^d$, let
\[
    P^c = \{ x \in P : \inner{c,x} \ge \inner{c,y} \text{ for all } y \in P\}
\]
be the face in direction $c$. Let $Q$ be a \Def{Minkowski summand} of $P$, i.e., $P=Q+R$ for some
  polytope~$R$. Then for every $c$ such that $P^c = \{p\}$ is a vertex, we have that $Q^c = \{q\}$
  is a vertex. This gives a surjective map $V(P) \to V(Q)$ with $p \mapsto q$. Moreover, if $e =
  pp'$ is an edge of $P$, then there is a scalar $\lambda_e(Q)\in[0,1]$ such that
\[
    q'-q = \lambda_e(Q)(p'-p).
\]
We define the \Def{edge-deformation vector} $\lambda(Q)=(\lambda_e(Q))_{e\in E}\in\RR^E$, which
  determines $Q$ uniquely up to translation. Indeed, for any two vertices $p,p' \in V(P)$, let $p =
  p_0p_1\dots p_k = p'$ be a path in the graph of $P$. Then
\begin{equation}\label{eqn:lamb_path}
    q'  \ = \ q + \sum_{i=1}^k \lambda_{p_{i-1}p_{i}} (p_i - p_{i-1}) \, .
\end{equation}
Realizing that this is independent of the chosen path proves the necessity of the following result
due to Shephard~\cite{Shephard}; see also~\cite[Theorem~15.1.2]{Gruenbaum}
and~\cite[Sect.~5.1]{Inscribed1}.

\begin{thm}\label{thm:shephard}
    Let $P \neq \emptyset$ be a polytope. Then $\lambda \in \RR^E$ is the
    edge-deformation vector of a summand of $P$ if and only if
    \begin{enumerate}[\rm (i)]
        \item $0\le \lambda_e\le 1$ for every $e\in E$, and
        \item for every $2$-face $F$ of $P$ with cyclically-ordered vertices $p_1,\dots,p_k$
            \begin{equation}\label{eq:cycle_equation}
                \lambda_{p_1p_2}(p_2-p_1)+\lambda_{p_2p_3}(p_3-p_2)+\cdots+\lambda_{p_kp_1}(p_1-p_k)=0.
	\end{equation}
    \end{enumerate}
\end{thm}

The theorem also shows that the set $\TP$ of all edge-deformation vectors of Minkowski summands of
$P$ is a convex polytope, that we call the \Def{deformation polytope} of $P$. Note that a proper polytope
$P$ is indecomposable if and only if $\TP$ is $1$-dimensional. That is, if for every Minkowski
summand $Q$ there is some $\mu\ge 0$ such that $\lambda(Q)=\mu \lambda(P)$.

Following~\cite{PP}, two edges $e,f\in E$ of $P$ are called \Def{dependent} if
$\lambda_e(Q)=\lambda_f(Q)$ for every Minkowski summand $Q$ of $P$, that is, $\TP \subseteq \{
\lambda : \lambda_e = \lambda_f \}$. This defines an equivalence relation $e \equiv f$ on $E$ that
characterizes indecomposability (c.f. {\cite[Lem.~2.2.5]{PP}}).

\begin{cor}\label{cor:equivalence}
    $P$ is indecomposable if and only if~$\equiv$ has a unique equivalence class.
\end{cor}

In \cite{PP}, several methods to prove edge dependencies are provided. We will only need the
following two (see {\cite[Ex.~2.2.2 and 2.2.3]{PP}}):

\begin{lem}\label{lem:triangleparallelogramdependent}
    If the edges $e,f$ belong to a common triangular $2$-face of $P$, or are opposite edges in a parallelogram $2$-face of $P$, then $e$ and $f$ are dependent.
\end{lem}

We will also use the following corollary of \cite[Thm.~2.5.4]{PP}.
\begin{thm}\label{thm:connected}
    If any pair of vertices of $P$ is connected through a path of pairwise dependent edges, then $P$ is indecomposable.
\end{thm}

And the following corollary of \cite[Thm.~2]{McMullen1987} (see also \cite[Sec.~2.6.3]{PP}).
\begin{lem}\label{lem:edge}
    If $P$ has an edge $e$ such that every facet of $P$ shares at least a vertex with $e$, then $P$ is indecomposable.
\end{lem}

\section{Deformation polytopes of $0/1$-polytopes}\label{sec:01-TP}

Throughout this section, let $P$ be a $0/1$-polytope of dimension $d \ge 2$. Since every
$k$-dimensional face of $P$ is linearly isomorphic to a $0/1$-polytope in $\RR^k$, every
$2$-dimensional face of $P$ is either a triangle or a parallelogram. It is shown in \cite[Cor.~4.2.6]{PP}
that for general polytopes satisfying this condition on 2-faces, the deformation cone is
simplicial. We give a streamlined version of the argument to show that the deformation polytopes
are cubes. Using the restricted geometry of 0/1-polytopes, this allows us to prove Theorem 1.1.
To that end, we define an equivalence relation on the edges $E$ of $P$ as the transitive 
closure of $e \sim f$ if $e,f$ belong to a common triangle, or are parallel edges of a 
parallelogram. For $S \subseteq E$, we denote its characteristic vector by $1_S \in \{0,1\}^E$. 
The following lemma, adapted from \cite[Lem.~4.2.5]{PP} follows directly by inspecting the conditions~\eqref{eq:cycle_equation} for triangles and parallelograms.

\begin{lem}\label{lem:equivclassinTP}
    Let $P$ be a $0/1$-polytope. If  $S$ is the $\sim$-equivalence class of some edge $e$ of $P$, then $1_S \in \TP$.
\end{lem}

This insight gives a precise description of the deformation polytope for $0/1$-polytopes. 
\begin{thm}\label{thm:TP_cube}
    Let $P$ be a $0/1$-polytope. Let $S_1,\dots,S_k$ be the equivalence classes of $\sim$. Then
    \[
        \TP \ = \ \sum_{i=1}^k [0,1_{S_i}] \, .
    \]
    In particular $\TP$ is linearly isomorphic to the $k$-dimensional unit cube $[0,1]^k$.
\end{thm}
\begin{proof}
    Let $Q$ be a Minkowski summand of $P$ with edge-deformation vector $\lamb = \lamb(Q) \in [0,1]^E$. We show that there are unique $\mu_1,\dots,\mu_k \in [0,1]$ such that
    \[
        \lamb \ = \
        \mu_1 1_{S_1} +
        \mu_2 1_{S_2} +
        \cdots +
        \mu_k 1_{S_k} \, .
    \]
    This will settle the first claim. Since the equivalence classes are pairwise disjoint, choosing
    a system of representatives $e_i \in S_i$, yields the linear isomorphism $\lambda \mapsto
    (\lambda_{e_i})_{i=1,\dots,k}$ and proves the second claim.

    To prove our claim, assume that there is $e \in E$ with $\lamb_e > 0$. If $f \in E$ satisfies
    $f \sim e$, then $e \equiv f$ by Lemma~\ref{lem:triangleparallelogramdependent}. Thus $\lamb_e
    = \lamb_f$ for all $f \sim e$. If $S_i$ is the equivalence class containing $e$, this shows
    that $\lamb' = \lamb - \lamb_e 1_{S_i}$ is non-negative. Moreover, $1_{S_i}$ satisfies all the
    equations in Theorem~\ref{thm:shephard}, by Lemma~\ref{lem:equivclassinTP}, and thus so does
    $\lamb' = \lamb - \lamb_e 1_{S_i}$. We conclude that $\lamb' \in \TP$, and the proof now
    follows by induction on the size of the support of $\lamb$.
\end{proof}

The proof extends to any polytope whose $2$-dimensional faces are triangles or parallelograms. 
 For polytopes affinely isomorphic to $0/1$-polytopes this condition
is satisfied. Notice that it is not enough to require that all $2$-faces are triangles or
\emph{quadrilaterals}.

\begin{quest}
    Is there a decomposable polytope that is combinatorially isomorphic to an indecomposable $0/1$-polytope?
\end{quest}

\begin{proof}[Proof of Theorem~\ref{thm:main}]
    Theorem~\ref{thm:TP_cube} implies that every $0/1$-polytope $P \subset \RR^d$ can be written as
    \[
        P \ = \ P_1 + P_2 + \cdots + P_k \, .
    \]
    where each $P_i$ is the Minkowski summand of $P$ whose edge-deformation vector is the
    characteristic vector $1_{S_i}$ of some equivalence class of edges. Since the edges in $S_i$
    are pairwise dependent, Corollary~\ref{cor:equivalence} shows that $P_i$ is indecomposable; and
    the product structure of $\TP$ shows that this decomposition is unique.

    We may assume that $0$ is a vertex of $P$ as well as of $P_1,\dots,P_k$. It follows that
    $P_1,\dots,P_k \subseteq P \subseteq [0,1]^d$. Since for every edge $e = pp'$ of $P$, the
    vector $p-p'$ has integer coordinates, it follows from \eqref{eqn:lamb_path} that each $P_i$ is
    itself a lattice polytope contained in the unit cube, that is, a $0/1$-polytope. 
    Thus, for any $1 \le i < j \le k$ and vertices $v \in V(P_i)$ and $v' \in V(P_j)$,
    we have $v,v' \in \{0,1\}^d$.
    Moreover, we have $v+v' \in
    \{0,1\}^d$, because $v+v'$ is a lattice point contained in~$P$. This implies that $v$ and $v'$ cannot have the
    same non-zero coordinate, and thus that $P_1,\dots,P_k$ lie in pairwise orthogonal coordinate
    subspaces, which completes the proof.
\end{proof}

\section{Indecomposability results for combinatorial classes of polytopes}\label{sec:applications}
In this section we discuss applications of Theorem~\ref{thm:main} to several classes of polytopes
from combinatorics, optimization, and algebraic geometry. The common theme is that a geometric
question about Minkowski decomposability is translated into a combinatorial question about product
structure. In the examples below indecomposability is governed by familiar connectivity conditions
on the underlying discrete objects. For several families we also give alternative direct proofs,
which make the connection between the geometry of the polytope and the underlying combinatorial
data more explicit.

\subsection{Simple and simplicial $0/1$-polytopes}

A polytope $P$ is \Def{simplicial} if every $k$-dimen\-sional face has exactly $k+1$ vertices for
$k < \dim P$. It is a classical result of Shephard~\cite[Thm.~13]{Shephard} that simplicial
polytopes are indecomposable; see also \cite[Sec.~15.1]{Gruenbaum}.

A polytope $P$ is \Def{simple} if every vertex is incident to exactly $\dim (P)$ many facets or,
equivalently, edges. Simple polytopes have the property that they maximize the dimension of the
deformation cone among all polytopes of fixed dimension $n$ and number of facets $m$, namely $m-n$.
Thus, except for the $n$-dimensional simplex, all simple polytopes of dimension $n$ are
decomposable. The following result was obtained by Kaibel and Wolff~\cite{Kaibel} for which we give
a short proof.

\begin{thm}
    If $P$ is a simple $0/1$-polytope, then $P$ is a Cartesian product of simplices.
\end{thm}
\begin{proof}
    If $P$ is simple, then either $P$ is a simplex or decomposable. If it is not a simplex, by Theorem~\ref{thm:main} there are $0/1$-polytopes
    $P_1, P_2$ with $P = P_1 \times P_2$. Now $P$ is simple if and only if $P_1$ and $P_2$ are simple and the proof follows by induction on the dimension.
\end{proof}

\subsection{$2$-level and compressed polytopes}

A lattice polytope $P \subset \RR^n$ is \Def{compressed} if for every facet $F$ and every
supporting hyperplane $H$ with $F = P \cap H$, there is at most one hyperplane $H'$ parallel to $H$
such that $H' \cap P$ contains a lattice point. Compressed polytopes were introduced in
\cite{Stanley-compressed} as those lattice polytopes all whose pulling triangulations are
unimodular. The definition given above is due to Sullivant~\cite{Sullivant}, who also proved that
every compressed polytope is affinely equivalent to a $0/1$-polytope. The notion of compressed
polytope is strongly related to that of \emph{$2$-levelness}; see~\cite{GPT,SWZ}. A polytope $P$ is
\Def{$2$-level} if for every facet-defining hyperplane $H$ there is a parallel hyperplane $H'$ such
that all vertices of $P$ lie in $H \cup H'$. Note that $P$ is not required to be a lattice
polytope.

If $P$ is compressed, then $P$ is $2$-level. We recall the argument that $2$-level polytope $P$ is
affinely isomorphic to a $0/1$-polytope, from which it follows that the vertices span an affine
lattice for which $P$ is compressed.
\begin{prop}
    If $P$ is $2$-level, then $P$ is affinely isomorphic to a $0/1$-polytope.
\end{prop}
\begin{proof}
    We may assume that $P \subset \RR^n$ is full-dimensional and that $0$ is a vertex of $P$. Let
    $a_1,\dots,a_n$ be linearly independent normals to facets that contain $0$. We can further assume
    that the maximal value of $x \mapsto \langle a_i,x \rangle$ over $P$ is $1$ for all $i$. Define a linear transformation $T : \RR^n \to \RR^n$ by $T(x)_i = \langle a_i,x \rangle$. The $2$-level property then implies that every vertex $v$ satisfies $T(v) \in \{0,1\}^n$ and hence $T(P)$ is a $0/1$-polytope.
\end{proof}

As decomposability is a property that is invariant under affine transformations, we obtain the
following result directly from Theorem~\ref{thm:main}.

\begin{thm}
    If $P$ is a $2$-level (resp.\ compressed) polytope, then $P$ is decomposable if and only if $P$
    is a product of $2$-level (resp.\ compressed) polytopes.
\end{thm}

\subsection{Permutation polytopes}
Let $H$ be a finite group of order $n = |H|$. We can view $H$ as acting on itself by permutations
and identify $H$ with a subgroup of the symmetric group $S_n$. Every $s \in S_n$ is identified with
a permutation matrix $M_s \in \{0,1\}^{n \times n}$ and the \Def{permutation polytope} of $H$ was
defined in~\cite{PermutationPolytopes} as
\[
    P(H) \ = \ \conv(M_h : h \in H) \, .
\]
By \cite[Thm.~3.5]{PermutationPolytopes} a permutation polytope $P(H)$ is (combinatorially) a
  Cartesian product if and only if $H = H_1 \times H_2$ as finite groups. From
  Theorem~\ref{thm:main} we obtain
\begin{thm}\label{thm:permutation}
    Let $H$ be a finite group. Then $P(H)$ is decomposable if and only if $H = H_1 \times H_2$ for proper subgroups $H_1,H_2 \subset H$.
\end{thm}

\newcommand\Birk{\mathrm{Birk}}%
For $H = S_n$, the permutation polytope $\Birk_n = P(S_n)$ is the famous \Def{Birkhoff polytope} of doubly stochastic matrices. It follows directly from Theorem~\ref{thm:permutation}
\begin{cor}
    For any $n \ge 1$, the Birkhoff polytope $\Birk_n$ is indecomposable.
\end{cor}

\subsection{Antiblocking polytopes}
A non-empty polytope $P$ is \Def{antiblocking} if $P \subset \RR_{\ge0}^n$ and for every $y \in P$
and $x \in \RR^n$ with $0 \le x_i \le y_i$ for all $i$, we have $x \in P$. Antiblocking polytopes
were introduced by Fulkerson~\cite{Fulkerson} as an embodiment of hereditary set systems or
simplicial complexes. Let $P \subset \RR^n$ be an antiblocking $0/1$-polytope ($0/1$-antiblocking
for short). Define $\Delta \ := \ \{ A \subseteq [n] : 1_A \in P \}$, where $1_A=\sum_{i\in A} e_i$
is the characteristic vector of $A$. Then $\Delta$ satisfies that $B \in \Delta$ for all $B$ such
that $B \subseteq A$ for some $A \in \Delta$. Conversely if $\Delta \subseteq 2^{[n]}$ is a
non-empty hereditary set family, that is, a \Def{simplicial complex}, then the polytope
\[
    P_\Delta \ := \ \conv( 1_A : A \in \Delta )
\]
is $0/1$-antiblocking.

A \Def{nonface} is a set $N \in 2^{[n]} \setminus \Delta$. An inclusion-minimal nonface, i.e., a
nonface $N$ such that $N\setminus v \in \Delta$ for all $v \in N$, is simply called a minimal
nonface.%
\begin{lem}\label{lem:nonface}
    Let $\Delta$ be a simplicial complex. Let $N$ be a  minimal nonface and $i,j \in N$ distinct
    elements. Then $\conv( 1_{N\setminus \{i,j\}}, 1_{N\setminus \{i\}}, 1_{N\setminus \{j\}})$ is a
    triangular face of $P_\Delta$.
\end{lem}

\begin{proof}
    Consider the vertices of $P_\Delta$ that maximize the linear function $c
    \in \RR^n$ with $c_a < 0 = c_i = c_j < c_b$ for all $a \in [n]\setminus N$
    and $b \in N \setminus \{i,j\}$.
\end{proof}

\newcommand{\Gex}{G_{\mathrm{ex}}}%
Note that if $P_1,P_2$ are antiblocking, then so is $P_1 \times P_2$. If $P_i = P_{\Delta_i}$,
then $P_1 \times P_2 = P_\Delta$ for the \Def{join} $\Delta = \Delta_1 \ast \Delta_2 = \{A_1 \uplus
A_2 : A_i \in \Delta_i \}$. For a simplicial complex define the \Def{exclusion graph}
$\Gex(\Delta)$ on $[n]$ with an edge $ij$ if there is a minimal nonface $N$ with $i,j \in N$.

\begin{lem}\label{lem:Gex}
    Let $\Delta$ be a simplicial complex. Then $P_\Delta$ is a product if and only if $\Gex(\Delta)$ is disconnected.
\end{lem}
\begin{proof}
    If $P_\Delta$ is a product then $\Delta = \Delta_1 \ast \cdots \ast
    \Delta_k$ for $k \ge 2$ subcomplexes $\Delta_i \subseteq 2^{V_i}$. Since for
    any $u \in V_i, v \in V_j$ for $i\neq j$, we have $\{u,v\} \in \Delta$, the
    graph $\Gex$ has connected components $V_1,\dots,V_k$. Conversely, if
    $V_1,\dots,V_k$ are the connected components of $\Gex$, then every nonface $N$ is a subset of some $V_i$. Hence if $\Delta_i \subset \Delta$ is the subcomplex induced on $V_i$, then $\sigma_1 \cup \cdots \cup \sigma_k \in \Delta$ for all $\sigma_i \in \Delta_i$. Thus $\Delta = \Delta_1 \ast \cdots \ast \Delta_k$.
\end{proof}

Lemma~\ref{lem:Gex} together with Theorem~\ref{thm:main} already implies the following result. We give a
short proof that highlights the interaction of $P_\Delta$ and $\Delta$.

\begin{thm}\label{thm:antiblock}
    Let $P = P_\Delta$ be a $0/1$-antiblocking polytope. Then $P$ is indecomposable if and only if $\Gex(\Delta)$ is connected.
\end{thm}
\begin{proof}
    If $\Gex(\Delta)$ is disconnected, then $\Delta$ is a join and $P_\Delta$ is a product, thus decomposable. Assume now that $\Gex(\Delta)$ is connected.

    If $A \in \Delta$ and $i\in A$, then $1_A1_{A\setminus i}$ is an edge of $P_\Delta$. We first
    prove that all the edges parallel to $1_A-1_{A\setminus i}$ are dependent to
    $1_{\{i\}}1_{\emptyset}$. If $|A|>1$, then there is some $j\in A\setminus i$ and
    $\conv(1_A,1_{A\setminus i},1_{A\setminus \{i,j\}},1_{A\setminus j})$ is a parallelogram 2-face
    of $P_\Delta$ by the antiblocking property. By Lemma~\ref{lem:triangleparallelogramdependent},
    $1_A1_{A\setminus i}$ and $1_{A\setminus j}1_{A\setminus \{i,j\}}$ are dependent. We conclude
    by induction on~$|A|$.

    Now, let $ij$ be an edge of $\Gex$, and let $N$ be a minimal nonface with $i,j\in N$. Then by
    Lemma~\ref{lem:nonface} $F = \conv(1_{N\setminus \{i,j\}}, 1_{N\setminus i}, 1_{N\setminus j})$
    is a triangular face of $P_\Delta$. Again by Lemma~\ref{lem:triangleparallelogramdependent},
    all three edges of $F$ are dependent. In particular they are dependent to the edges
    $1_{\{i\}}1_{\emptyset}$ and $1_{\{j\}}1_{\emptyset}$. By the connectivity of $\Gex$, any two
    edges that are parallel to the coordinate axis are pairwise dependent.

    Note that every vertex $1_A$ with $A=\{i_1,\dots,i_k\}$ is connected to $0=1_\emptyset$ by a
    path of such edges: $1_A,1_{A\setminus i_1},\dots,1_{\{i_k\}},1_{\emptyset}$. By
    Theorem~\ref{thm:connected}, $P_\Delta$ is indecomposable.
\end{proof}

\newcommand{\oG}{\overline{G}}%
\newcommand{\oE}{\overline{E}}%
\subsection{Clique, stable sets, matching and edge polytopes}
Let $G = (V,E)$ be a simple undirected graph. A set $S \subseteq V$ is \Def{stable} if there is no
edge $ab \in E$ with $a,b \in S$. Conversely, if $ab \in E$ for every distinct $a,b \in S$, the set
$S$ is called a \Def{clique} of $G$. If we denote by $\oG = (V,\oE)$ the \Def{complement graph} of
$G$ with $\oE = \{ ab : a,b \in V, a \neq b, ab \not\in E \}$, then cliques of $G$ are stable sets
of $\oG$ and conversely. The \Def{stable set polytope} of $G$ is
\[
    \Stab(G) \ := \ \conv\{ 1_S : S \subseteq V \text{ stable}\} \, .
\]
The \Def{clique polytope} can then be defined as $\Cli(G) = \Stab(\oG)$. Stable set and clique
  polytopes are among the basic objects of polyhedral combinatorics~\cite{Schrijver}. They encode
  fundamental optimization problems while their linear descriptions and facial structure reflect
  subtle graph-theoretic properties.

Note that both $\Stab(G)$ and $\Cli(G)$ are $0/1$-antiblocking with underlying simplicial complexes
being the collections of stable sets and cliques, respectively. Since in this case minimal nonfaces
are edges (nonedges, respectively), we infer that exclusion graphs are $G$ ($\oG$, respectively).

\begin{thm}\label{thm:stab}
    Let $G$ be a simple graph. Then $\Stab(G)$ (resp. $\Cli(G)$) is indecomposable if and only if $G$ (resp. $\oG$) is connected.
\end{thm}

A \Def{matching} of a graph $G = (V,E)$ is a collection of edges $M \subseteq E$ such that any two
distinct $e,f \in M$ are not incident to a common vertex. We may assume that $G$ does not have
isolated vertices and at least one edge. The collection of matchings of $G$ is a simplicial complex
and the associated $0/1$-antiblocking is the \Def{matching polytope} $\Mat(G)$ of $G$. The
\Def{line graph} $L(G)$ of $G$ is the graph with $E$ as vertices whose edges $ef$ are precisely
those pairs of edges of $G$ that share a common endpoint. Hence matchings in $G$ are stable sets in
$L(G)$. With our assumptions, we note that $L(G)$ is connected if and only if $G$ is.
\begin{cor}
    The matching polytope $\Mat(G)$ is indecomposable if and only if $G$ is connected.
\end{cor}

\subsection{Pure simplicial complexes: matroid and edge polytopes}
\newcommand{\Fac}{\mathcal{A}}%
A simplicial complex $\Delta \subset 2^{[n]}$ is called \Def{pure} if every inclusion-maximal $A
\in \Delta$ has the same cardinality $d$. The collection of inclusion-maximal elements
$\Fac = \Fac(\Delta)$ are the vertices of the face
\[
    P_\Fac \ = \ \{ x \in P_\Delta : x_1 + \cdots + x_n = d \} \ = \ \conv(1_A : A \in \Fac) \, .
\]
of $P_\Delta$, which completely determines $P_\Delta$.

It is straightforward to see that $P_\Fac$ is a product if and only if $P_\Delta$ is a product and
Theorem~\ref{thm:main} yields the following.

\begin{prop}
    For a pure simplicial complex, $P_\Fac$ is indecomposable if and only if $P_\Delta$ is indecomposable.
\end{prop}

\newcommand{\Ind}{\mathcal{I}}%
\newcommand{\Bases}{\mathcal{B}}%
Matroids are combinatorial abstractions of linear independence that play a
central role in optimization, tropical geometry, and combinatorial Hodge
theory~\cite{AHK}. Their base polytopes translate matroid structure into convex geometry.
Here Theorem~\ref{thm:main} shows that this translation preserves irreducibility in that
Minkowski indecomposability of the base polytope is equivalent to connectedness
of the matroid.

A matroid $M$ can be described as a pure simplicial complex $\Ind(M) \subseteq 2^E$ satisfying the
augmentation property. This is the \Def{independence complex} of $M$ and $P_M = P_{\Ind(M)}$ is the
\Def{independence polytope}. The \Def{bases} $\Bases(M)$ of $M$ are the inclusion-maximal
independent sets and $B_M = P_{\Bases}$ is the \Def{matroid base polytope} of $M$. The matroid $M$
is \Def{connected} if and only if $\Ind(M)$ is not a join of simplicial complexes.

\begin{cor}[\cite{Ngu86,LPP}]
    Let $M$ be a matroid. Then $B_M$ is indecomposable if and only if $P_M$ is indecomposable if and only if $M$ is connected.
\end{cor}

Let $G = ([n],E)$ be a simple connected graph with $E \neq \emptyset$. Ohsugi and Hibi~\cite{OH}
introduced \Def{edge polytopes}
\[
    \EP_G \ := \ \conv( e_i + e_j : ij \in E)
\]
in the context of toric geometry and commutative algebra. Edge polytopes have been well studied
  in relation to fundamental algebraic-geometric properties such as normality of the associated
  edge ring (Ehrhart ring), equivalently the existence of unimodular coverings/triangulations;
  see~\cite{BinomialIdeals}. The edge polytope of $G$ is the face $P_{\Fac(\Delta)}$ of the
  simplicial complex $\Delta = \{\emptyset\} \cup V \cup E$.

\begin{cor}
    The edge polytope $\EP_G$ of a simple graph $G$ is indecomposable if and only if $G$ is connected.
\end{cor}

\newcommand\Gco{\mathrm{Comp}}%
\subsection{Order and chain polytopes} \label{sec:order}%
Let $\poset = ([n],\preceq)$ be a finite partially ordered set. The \Def{order polytope},
introduced by Geissinger~\cite{Geissinger} and studied in depth by Stanley~\cite{S86}, is the
polytope of order preserving maps from $\poset$ to $[0,1]$
\[
    \Oc(\poset) \ = \ \{ x \in [0,1]^n : x_a \le x_b \text{ for all } a \prec b \} \, .
\]
Similar to antiblocking polytopes of simplicial complexes, order polytopes geometrically encode
  posets. The vertices of $\Oc(\poset)$ are precisely the indicator functions $1_F \in \{0,1\}^n$
  of filters of $\poset$, that is, sets $F \subseteq [n]$ such that for all $a \in F$ and $b
  \succeq a$ we have $b \in F$. The \Def{comparability graph} $\Gco(\poset)$ is the undirected
  simple graph on $[n]$ with edges $ab$ if $a \prec b$ or $b \prec a$. It is straightforward to
  verify that $\Oc(\poset)$ is a product if and only if $\Gco(\poset)$ is disconnected. The factors
  of $\Oc(\poset)$ are order polytopes of the subposets induced on the connected components of
  $\Gco(\poset)$.

\begin{thm}\label{thm:order_polytope}
    Let $\poset$ be a finite poset. Then $\Oc(\poset)$ is indecomposable if and only if $\Gco(\poset)$ is connected.
\end{thm}

The result again follows directly from Theorem~\ref{thm:main} but we give two proofs that, similar
to $0/1$-antiblocking polytopes, ties geometry to combinatorics. They have the additional property
that they only rely on the combinatorial type of the polytope and not the geometric realization,
and thus are valid for any combinatorially equivalent polytope.

For two filters $F,G$ the vertices $1_F,1_G$ span an edge if and only if, say, $F \subseteq G$ and
$G \setminus F$ has a connected comparability graph.

The first proof extends the proof of indecomposability of \Def{shard polytopes}
from~\cite[Prop.~64]{PPR}. \emph{Shard polytopes} are a special subfamily of order polytopes that
are also the matroid polytopes for certain series-parallel graphic matroids, and also isomorphic to
the matroid polytopes of certain positroids, namely those arising as lattice path matroids of
snakes. Shard polytopes were introduced in \cite{PPR} as building blocks for constructing
quotientopes via Minkowski sums.

\begin{proof}[{First proof of Theorem~\ref{thm:order_polytope}}]
    If $\Gco(\poset)$ is not connected, then $\Oc(\poset)$ is a product and thus decomposable. If $\Gco(\poset)$ is connected, then $1_\poset,1_\emptyset$ span an edge of $\Oc(\poset)$ as discussed above.
    Moreover, every facet of $\Oc(\poset)$ contains either $1_\emptyset=(0,\ldots,0)$ or $1_\poset=(1,\ldots,1)$. Indeed, there are three types of inequalities defining the facets of $\Oc(\poset)$:
    \begin{itemize}
        \item[(i)] If $p_i \in \poset$ is minimal, then $x_i = 0$ is a supporting hyperplane;
        \item[(ii)] If $p_i \in \poset$ is maximal, then $x_i =1$ is a supporting hyperplane;
        \item[(iii)] For $p_i,p_j \in \poset$, if $p_i$ is covered by $p_j$, then $x_i \leq x_j$ is a supporting hyperplane.
    \end{itemize}
    The facets of type (i) contain~$1_\emptyset$, those of type (ii) contain~$1_\poset$ and those of type (iii) contain both.
    We conclude by Lemma~\ref{lem:edge}.
\end{proof}

The second proof will uses the sufficiency of Lemma~2.2 from~\cite{Mori}. We give a short proof for
completeness.
\begin{lem}\label{lem:ord_triangles}
    Let $\poset$ be a connected poset and filters $F_1 \subset F_2 \subset F_3$
    such that $F_j \setminus F_i$ is connected for all $i<j$. Then $\conv(
    1_{F_1},
    1_{F_2},
    1_{F_3})$ is a triangular face of $\Oc(\poset)$.
\end{lem}
\newcommand\cov{\prec\!\,\!\!\!{\mathbin{\raisebox{0.35ex}{\scalebox{0.5}{$\bullet$}}}}\ }%
\begin{proof}
    For a cover relation $a \cov b$, consider the linear function $\ell_{ab}(x)
    = x_a - x_b$. Then for every filter $F$ we have $\ell(1_F) = 0$ if $a \in F$ or $b \not\in F$ and $< 0$ otherwise. Let $S$ be the collection of
    cover relations $a \cov b$ such that $a,b \in F_j \setminus F_i$ for some $i < j$. The linear function
    \[
        \ell(x) \ = \ \sum_{ab \in S} \ell_{ab}(x) - \sum_{a \not\in F_3} x_a
    \]
    satisfies $\ell(1_G) \le 0$ on all filters of $\poset$ and, as the differences are connected,
    $\ell(1_G)=0$ precisely when $G = F_i$ for $i=1,2,3$.
\end{proof}

\begin{proof}[Second proof of Theorem~\ref{thm:order_polytope}]
    We have to show that if $\Gco(\Pi)$ is connected, then $\Oc(\Pi)$ is indecomposable. By
    assumption $1_\emptyset 1_{\poset}$ is an edge of $\Oc(\poset)$. We show that every edge $1_F
    1_{F'}$ is dependent to $1_\emptyset 1_{\poset}$. The result then follows from
    Corollary~\ref{cor:equivalence}.

    Arguing by contradiction, let $1_F 1_{F'}$ be an edge that is not dependent to $1_\emptyset
    1_\Pi$ with $|F' \setminus F|$ maximal. Since $F$ is a proper subset of $F'$, $N := F'
    \setminus F$ is non-empty, and $N \neq \Pi$ because $1_F 1_{F'}\neq 1_\emptyset 1_\Pi$ . As
    $\Gco(\Pi)$ is connected, there is a path $s_0s_1 \dots s_k$ starting in some $s_0 \in \Pi
    \setminus N$ and ending in $s_k \in N$. Let $i \ge 0$ be minimal with $s_{i+1} \in N$. There
    are two cases.

    If $s_i \not\in F'$, then define the filter $G := F' \cup \Pi_{\succeq s_{i}}$. By definition
    $G$ contains $F'$ and $G \setminus F' = \Pi_{\succeq s_{i}} \setminus F'$ has $s_{i}$ as a
    minimal element and hence is connected. To prove that $G \setminus F = N \cup (\Pi_{\succeq
    s_{i}} \setminus F')$ is also connected, we note that $N$ and $\Pi_{\succeq s_{i}} \setminus
    F'$ are both connected and the edge $s_{i}s_{i+1}$ in $\Gco(\Pi)$ shows that $G \setminus F$ is
    connected. Lemma~\ref{lem:ord_triangles} now implies that $1_F 1_G$ is an edge that is also
    dependent to $1_F 1_{F'}$ and $|G \setminus F| > |F' \setminus F|$. A contradiction.

    If $s_i \in F$, then the filter $H := F \setminus \Pi_{\preceq s_{i}}$ is a proper subset of
    $F$ and $F \setminus H = F \cap \Pi_{\preceq s_{i}}$ is connected because it has a maximal
    element. The same argument as above shows that $F' \setminus H = N \cup F \cap \Pi_{\preceq
    s_{i}}$ is connected. Lemma~\ref{lem:ord_triangles} again implies that $1_H1_{F'}$ is not
    dependent to $1_\emptyset 1_\Pi$ but $|F' \setminus H| > |F' \setminus F|$. Also a
    contradiction.
\end{proof}

Stanley introduced another polytope associated to a poset. The \Def{chain polytope} $\Cc(\poset)$
is the set of all $x \in [0,1]^n$ such that
\[
    x_{a_1} +
    x_{a_2} +
    \cdots +
    x_{a_k}  \ \le \ 1
\]
for all chains $ a_1 \prec a_2 \prec \dots \prec a_k$. The vertices of $\Cc(\poset)$ are in
  bijection to antichains in $\poset$. Chains and antichains are exactly the cliques and stable
  sets of the comparability graph $\Gco(\poset)$ and it can be shown that $\Cc(\poset) =
  \Stab(\Gco(\poset))$. Hence, from Theorem~\ref{thm:stab} we obtain the following consequence.

\begin{cor}\label{cor:chain}
    Let $\poset$ be a finite poset. The chain polytope $\Cc(\poset)$ is indecomposable if and only if $\Gco(\poset)$ is connected.
\end{cor}

\subsection{Flow polytopes}\label{sec:flow}%
Flow polytopes are lattice polytopes whose geometry is controlled by the combinatorics of a
directed graph. They appear, for example, in the study of Kostant partition functions and diagonal
harmonics; see~\cite{BV,LMM}.

Let $D = (V,E)$ be a directed graph. We assume that $D$ does not have directed cycles and there is
a unique source $s$ and a unique sink $t$. For $b \in \RR^V$, a \Def{$b$-flow}, or flow with
netflow vector $b$, is an assignment $f: E \to \RR_{\ge0}$ such that for every node $v \in V$
\[
    \sum_{vu \in E} f(v,u) -
    \sum_{uv \in E} f(u,v)  \ = \ b_v \, .
\]
The collection of $b$-flows constitutes the \Def{flow polytope} $\Flow_D(b) \subset \RR^E$. In
  geometric and algebraic combinatorics, flow polytopes are typically considered for the vector
  $b^0$ with $-b^0_s = b^0_t = 1$ and $b^0_v = 0$ for $v \in V \setminus\{s,t\}$. We will simply
  write $\Flow_D = \Flow_D(b^0)$. In this case, the vertices of $\Flow_D$ are of the form $e_\rho
  \in \{0,1\}^E$, where $\rho$ is a directed $s$--$t$-path. An $s$--$t$-separator is a set $S
  \subseteq V \setminus \{s,t\}$ such that the removal of $S$ leaves no directed $s$--$t$-paths. If
  $D$ has a separator $S = \{r\}$, $D$ can be decomposed into two digraphs $D_1$ and $D_2$ with
  source-sink pairs $(s,r)$ and $(r,t)$. It is straightforward to verify that $\Flow_D =
  \Flow_{D_1} \times \Flow_{D_2}$. Conversely, if $\Flow_D$ is a Cartesian product, then $D$ must
  have an $s$--$t$-separator of size $1$.

\begin{thm}
    Let $D$ be an acyclic digraph with unique source $s$ and sink $t$. Then $\Flow_D$ is indecomposable if and only if $D$ does not have an $s$--$t$-separator of size $1$.
\end{thm}

\bibliographystyle{alpha}
\bibliography{Biblio.bib}

\newcommand{\etalchar}[1]{$^{#1}$}
\begin{thebibliography}{CDG{\etalchar{+}}22}

\bibitem[ACEP20]{Ardila}
Federico Ardila, Federico Castillo, Christopher Eur, and Alexander Postnikov.
\newblock Coxeter submodular functions and deformations of {Coxeter}
  permutahedra.
\newblock {\em Adv. Math.}, 365:36, 2020.
\newblock Id/No 107039.

\bibitem[AHK18]{AHK}
Karim Adiprasito, June Huh, and Eric Katz.
\newblock Hodge theory for combinatorial geometries.
\newblock {\em Ann. Math. (2)}, 188(2):381--452, 2018.

\bibitem[Alt95]{Altmann}
Klaus Altmann.
\newblock Minkowski sums and homogeneous deformations of toric varieties.
\newblock {\em T{\^o}hoku Math. J. (2)}, 47(2):151--184, 1995.

\bibitem[BHNP09]{PermutationPolytopes}
Barbara Baumeister, Christian Haase, Benjamin Nill, and Andreas Paffenholz.
\newblock On permutation polytopes.
\newblock {\em Adv. Math.}, 222(2):431--452, 2009.

\bibitem[BV08]{BV}
Welleda Baldoni and Mich{\`e}le Vergne.
\newblock Kostant partitions functions and flow polytopes.
\newblock {\em Transform. Groups}, 13(3-4):447--469, 2008.

\bibitem[CDG{\etalchar{+}}22]{DefConeCubes}
Federico Castillo, Joseph Doolittle, Bennet Goeckner, Michael~S. Ross, and
  Li~Ying.
\newblock Minkowski summands of cubes.
\newblock {\em Bull. Lond. Math. Soc.}, 54(3):996--1009, 2022.

\bibitem[CLS11]{CoxLittleSchenckToric}
David~A. Cox, John~B. Little, and Henry~K. Schenck.
\newblock {\em Toric varieties}, volume 124 of {\em Graduate Studies in
  Mathematics}.
\newblock American Mathematical Society, Providence, RI, 2011.

\bibitem[Fou16]{Fourier2016}
Ghislain Fourier.
\newblock Marked poset polytopes: {M}inkowski sums, indecomposables, and
  unimodular equivalence.
\newblock {\em J. Pure Appl. Algebra}, 220(2):606--620, 2016.

\bibitem[Ful72]{Fulkerson}
D.~R. Fulkerson.
\newblock Anti-blocking polyhedra.
\newblock {\em J. Comb. Theory, Ser. B}, 12:50--71, 1972.

\bibitem[Gal54]{Gale1954}
David Gale.
\newblock Irreducible convex sets.
\newblock In {\em Proc. of the {ICM}, {A}msterdam, 1954, {V}ol. 2}, pages
  217--218. Erven P. Noordhoff N. V., Groningen, 1954.

\bibitem[Gei81]{Geissinger}
Ladnor Geissinger.
\newblock The face structure of a poset polytope.
\newblock Combinatorics and computing, {Proc}. 3rd {Caribb}. {Conf}., {Cave}
  {Hill}/ {Barbados} 1981, 125-133 (1981)., 1981.

\bibitem[GPT10]{GPT}
Jo{\~a}o Gouveia, Pablo~A. Parrilo, and Rekha~R. Thomas.
\newblock Theta bodies for polynomial ideals.
\newblock {\em SIAM J. Optim.}, 20(4):2097--2118, 2010.

\bibitem[Gr{\"u}03]{Gruenbaum}
Branko Gr{\"u}nbaum.
\newblock {\em Convex polytopes. {Prepared} by {Volker} {Kaibel}, {Victor}
  {Klee}, and {G{\"u}nter} {M}. {Ziegler}}, volume 221 of {\em Grad. Texts
  Math.}
\newblock New York, NY: Springer, 2nd ed. edition, 2003.

\bibitem[HHO18]{BinomialIdeals}
J{\"u}rgen Herzog, Takayuki Hibi, and Hidefumi Ohsugi.
\newblock {\em Binomial ideals}, volume 279 of {\em Grad. Texts Math.}
\newblock Cham: Springer, 2018.

\bibitem[Kal82]{Kallay1982}
Michael Kallay.
\newblock Indecomposable polytopes.
\newblock {\em Israel J. Math.}, 41(3):235--243, 1982.

\bibitem[KW00]{Kaibel}
Volker Kaibel and Martin Wolff.
\newblock Simple 0/1-polytopes.
\newblock {\em Eur. J. Comb.}, 21(1):139--144, 2000.

\bibitem[LMM19]{LMM}
Ricky~Ini Liu, Alejandro~H. Morales, and Karola M{\'e}sz{\'a}ros.
\newblock Flow polytopes and the space of diagonal harmonics.
\newblock {\em Can. J. Math.}, 71(6):1495--1521, 2019.

\bibitem[LPP25]{LPP}
Georg Loho, Arnau Padrol, and Germain Poullot.
\newblock Many rays of the submodular cone, 2025.
\newblock arXiv preprint
  \href{https://arxiv.org/abs/2510.03177}{\texttt{arXiv:2510.03177}}.

\bibitem[McM73]{McMullen-reps}
P.~McMullen.
\newblock Representations of polytopes and polyhedral sets.
\newblock {\em Geom. Dedicata}, 2:83--99, 1973.

\bibitem[McM87]{McMullen1987}
Peter McMullen.
\newblock Indecomposable convex polytopes.
\newblock {\em Israel J. Math.}, 58(3):321--323, 1987.

\bibitem[McM93]{McMullen-simple}
Peter McMullen.
\newblock On simple polytopes.
\newblock {\em Invent. Math.}, 113(2):419--444, 1993.

\bibitem[Mor25]{Mori}
Aki Mori.
\newblock Triangular faces of the order and chain polytope of a maximal ranked
  poset.
\newblock {\em Discrete Math.}, 348(8):6, 2025.
\newblock Id/No 114480.

\bibitem[MS24]{Inscribed1}
Sebastian Manecke and Raman Sanyal.
\newblock Inscribable fans i: Inscribed cones and virtual polytopes.
\newblock {\em Mathematika}, 70(4):e12270, 2024.

\bibitem[Ngu86]{Ngu86}
Hien~Quang Nguyen.
\newblock Semimodular functions.
\newblock 26:272--297, 1986.

\bibitem[OH98]{OH}
Hidefumi Ohsugi and Takayuki Hibi.
\newblock Normal polytopes arising from finite graphs.
\newblock {\em J. Algebra}, 207(2):409--426, 1998.

\bibitem[Pos09]{Postnikov}
Alexander Postnikov.
\newblock Permutohedra, associahedra, and beyond.
\newblock {\em Int. Math. Res. Not. IMRN}, (6):1026--1106, 2009.

\bibitem[PP26]{PP}
Arnau Padrol and Germain Poullot.
\newblock The graph of implicit edge dependencies for indecomposability and
  beyond, 2026.
\newblock arXiv preprint
  \href{https://arxiv.org/abs/2512.05307}{\texttt{arXiv:2512.05307}}.

\bibitem[PPR23]{PPR}
Arnau Padrol, Vincent Pilaud, and Julian Ritter.
\newblock Shard polytopes.
\newblock {\em Int. Math. Res. Not. IMRN}, (9):7686--7796, 2023.

\bibitem[PRW08]{PostnikovReinerWilliams}
Alexander Postnikov, Victor Reiner, and Lauren~K. Williams.
\newblock Faces of generalized permutohedra.
\newblock {\em Doc.~Math.}, 13:207--273, 2008.

\bibitem[PY08]{PrzeslawskiYost2008}
Krzysztof Przes\l{}awski and David Yost.
\newblock Decomposability of polytopes.
\newblock {\em Discrete Comput. Geom.}, 39(1-3):460--468, 2008.

\bibitem[PY16]{PrzeslawskiYost2016}
Krzysztof Przes\l{}awski and David Yost.
\newblock More indecomposable polyhedra.
\newblock {\em Extracta Math.}, 31(2):169--188, 2016.

\bibitem[RW73]{RosenmullerWeidner1973-ExtremeConvexSetFunctions}
Joachim Rosenm\"uller and Hans~G{\" u}nther Weidner.
\newblock A class of extreme convex set functions with finite carrier.
\newblock {\em Advances in Math.}, 10:1--38, 1973.

\bibitem[Sch03]{Schrijver}
Alexander Schrijver.
\newblock {\em Combinatorial optimization. {Polyhedra} and efficiency (3
  volumes)}, volume~24 of {\em Algorithms Comb.}
\newblock Berlin: Springer, 2003.

\bibitem[She63]{Shephard}
Geoffrey~C. Shephard.
\newblock Decomposable convex polyhedra.
\newblock {\em Mathematika}, 10:89--95, 1963.

\bibitem[Sta80]{Stanley-compressed}
Richard~P. Stanley.
\newblock Decompositions of rational convex polytopes.
\newblock Ann. {Discrete} {Math}. 6, 333-342 (1980)., 1980.

\bibitem[Sta86]{S86}
Richard~P. Stanley.
\newblock Two poset polytopes.
\newblock {\em Discrete Comput. Geom.}, 1(1):9--23, 1986.

\bibitem[Sul06]{Sullivant}
Seth Sullivant.
\newblock Compressed polytopes and statistical disclosure limitation.
\newblock {\em T{\^o}hoku Math. J. (2)}, 58(3):433--445, 2006.

\bibitem[SWZ09]{SWZ}
Raman Sanyal, Axel Werner, and G{\"u}nter~M. Ziegler.
\newblock On {Kalai}'s conjectures concerning centrally symmetric polytopes.
\newblock {\em Discrete Comput. Geom.}, 41(2):183--198, 2009.

\bibitem[Zie95]{ziegler}
G{\"u}nter~M. Ziegler.
\newblock {\em Lectures on polytopes}, volume 152 of {\em Grad. Texts Math.}
\newblock Berlin: Springer-Verlag, 1995.

\end{thebibliography}
\end{document}